\documentclass[10pt]{article}
\usepackage{psfrag,amsmath}
\usepackage{amssymb}
\usepackage{mathrsfs}
\usepackage{amsfonts}
\usepackage{amsmath}
\usepackage{mathrsfs,color}
\usepackage{epstopdf}
\usepackage{graphicx}
\usepackage{mathrsfs,amscd,amssymb,amsthm,amsmath,bm,url}
\usepackage{tikz}
\usetikzlibrary{positioning,backgrounds}
\usetikzlibrary{fadings}
\usetikzlibrary{patterns}
\usetikzlibrary{calc}
\usetikzlibrary{shadings}
\pgfdeclarelayer{background}
\pgfdeclarelayer{foreground}
\pgfsetlayers{background,main,foreground}

\makeatletter

\renewcommand{\@seccntformat}[1]{{\csname the#1\endcsname}{\normalsize .}\hspace{.5em}}
\makeatother

\def \[{\begin{equation}}
\def \]{\end{equation}}

\newtheorem{thm}{Theorem}[section]

\newtheorem{fact}{Fact}
\newtheorem{lem}[thm]{Lemma}

\newtheorem{conj}[thm]{Conjecture}

\newenvironment{wst}
{\setlength{\leftmargini}{1.5\parindent}
 \begin{itemize}
 \setlength{\itemsep}{-1.1mm}}
{\end{itemize}}
\begin{document}

\setlength{\baselineskip}{0.20in}
\begin{center}{\Large \bf A  sharp lower bound on the generalized 4-independence number\footnote{J.H.  financially supported by the National Natural Science Foundation of China (Grant No. 12001202), the Guangdong Basic and Applied Basic Research Foundation (Grant No. 2023A1515010838) and the Guangzhou Basic and Applied Basic Research Foundation (Grant No. 2024A04J3328).}}
\vspace{4mm}

{\large Jing Huang\footnote{Corresponding author. \\
\hspace*{5mm}{\it Email addresses}: jhuangmath@foxmail.com (J. Huang).}\vspace{2mm}}

School of Mathematics and Information Science, Guangzhou University, Guangzhou, 510006, China
\end{center}


\begin{abstract}
For a graph $G$, a vertex subset $S$ is called a maximum generalized $k$-independent set if the induced subgraph $G[S]$ does not contain a $k$-tree as its subgraph, and the subset
has maximum cardinality. The generalized $k$-independence number  of $G$, denoted as $\alpha_k(G)$,  is the number of vertices in a  maximum generalized $k$-independent set of $G$. For a graph $G$ with $n$ vertices, $m$ edges, $c$ connected components, and $c_1$ induced cycles of length 1 modulo 3,
Bock et al. [J. Graph Theory 103 (2023) 661-673]
showed that $\alpha_3(G)\geq n-\frac{1}{3}(m+c+c_1)$
and   identified the extremal graphs in which every two cycles are vertex-disjoint.
Li and Zhou [Appl. Math. Comput. 484 (2025) 129018] proved that if $G$ is a tree with $n$ vertices, then $\alpha_4(G) \geq \frac{3}{4}n$. They also presented all the  corresponding extremal trees.
In this paper, for a general graph $G$
with $n$ vertices, it is   proved that $\alpha_4(G)\geq \frac{3}{4}(n-\omega(G))$  by using a  different approach, where $\omega(G)$ denotes
the dimension of the cycle space of $G$. The  graphs whose generalized $4$-independence number attains the lower bound are characterized completely. This represents a logical continuation of the  work by Bock et al. and serves as a natural extension of the result by Li and Zhou.

\end{abstract}

\vspace{2mm} \noindent{\bf Keywords}: Generalized $k$-independence number; Dimension of the cycle space.

\setcounter{section}{0}

\section{\normalsize Introduction}\setcounter{equation}{0}
We start with introducing some background information that leads to our main results. Our main results will also be given in this section.
\subsection{\normalsize Background and definitions}
Let $G=(V_G, E_G)$ be a graph with vertex set $V_G$ and edge set $E_G$. The path, cycle, star,  and complete graph of order $n$ are conventionally
represented by  $P_n, C_n, S_n$ and  $K_n$, respectively.
For a vertex $v\in V_G,$ let $N_G(v)$ be the neighborhood of $v$, and
$N_G[v]=N_G(v)\cup \{v\}$ be the closed neighborhood of $v$.
We call $d_G(v):=|N_G(v)|$ the \textit{degree} of $v$.
A vertex of a graph $G$ is called a \textit{pendant vertex} if it is a vertex with degree one in $G$, whereas a vertex of $G$ is called a \textit{quasi-pendant vertex} if it is adjacent to a pendant vertex in $G$.
Unless stated otherwise, we adhere to the notation and terminology in \cite{C95}.

Denote by $\omega(G)$ the \textit{dimension} of the cycle space of $G$, that is $\omega(G)=|E_G|-|V_G|+c(G)$, where $c(G)$ is the number of connected components of $G$. A simple graph $G$ is called  \textit{acyclic} if it contains no cycles, whereas it is called an \textit{empty graph} if it has no edges.
For an induced subgraph $H$ of $G$, $G-H$ is the subgraph obtained from $G$ by deleting
all vertices of $H$ and all incident edges. For $W\subseteq V_G, G-W$ is the subgraph obtained from $G$ by deleting all
vertices in $W$ and all incident edges. For the sake of simplicity, we use $G-v, G-uv$ to denote the graph obtained from $G$ by deleting vertex $v \in V_G$, or edge $uv \in E_G$, respectively.
For two graphs $G_1$ and $G_2$, denote by $G_1 \cup G_2$ the  disjoint union of $G_1$ and $G_2$.
 For simplicity, we use $kG$ to denote the disjoint union of $k$ copies of $G$.

For a positive integer $k\geq 2$ and a subset $S\subseteq V_G$, we call $S$ a
generalized $k$-independent set if the induced subgraph $G[S]$ does not
contain a $k$-tree (a tree with $k$ vertices) as a subgraph. A maximum
generalized $k$-independent set of $G$ is a generalized $k$-independent
set with the maximum cardinality. The generalized $k$-independence number of
$G$, written as $\alpha_k(G)$, is the cardinality of a maximum
generalized $k$-independent set of $G$. Clearly, $\alpha_2(G)$ is
exactly the \textit{independence number}, whereas $\alpha_3(G)$ is called
the \textit{dissociation number} which was originally raised
by Yannakakis \cite{Y81} in 1981, and in the same paper, he showed
the  problem of computing dissociation number is NP-complete for bipartite graphs.
Cameron and Hell \cite{C06} showed that the problem can be solved in
polynomial time for some important classes of graphs such as
chordal graphs, weakly chordal graphs, asteroidal triple-free graphs,
and interval-filament graphs.
For more algorithms on the  dissociation number  one may be
referred to \cite{A07,B04,HB22,KK11,O11,T19}.

The problem concerning the bound of dissociation number in a given class of graphs is a
classical  problem and has been extensively studied. In 2009,  G\"{o}ring et al. \cite{GH09}
showed that
$$
\alpha_3(G)\geq\sum_{u\in V_G}\frac{1}{d_G(u)+1}+\sum_{uv\in E_G}\binom{|N_G[u]\cup N_G[v]|}{2}^{-1}.
$$
For a graph $G$ with $n$ vertices and
$m$ edges, Bre\v{s}ar et al. \cite{B11} pointed  out that $\alpha_3(G)\geq\frac{2n}{3}-\frac{m}{6}.$
Furthermore, they also proved that
\begin{equation*}
    \alpha_3(G)\geq
    \left\{
    \begin{array}{ll}
        \frac{n}{\left\lceil\frac{\Delta+1}{2}\right\rceil},& \textrm{if $G$ has maximum degree $\Delta,$}\\[8pt]
        \frac{4}{3}\sum_{u\in V_G}\frac{1}{d_G(u)+1},& \textrm{if $G$ has no isolated vertex,}\\[5pt]
        \frac{n}{2},& \textrm{if $G$ is outerplanar,}\\[5pt]
        \frac{2n}{3},& \textrm{if $G$ is a tree.}
    \end{array}
    \right.
\end{equation*}
Bre\v{s}ar et al. \cite{BJ13}  further  demonstrated that $\alpha_3(G)\geq\frac{2n}{k+2}-\frac{m}{(k+1)(k+2)},$ where
$k=\left\lceil\frac{m}{n}\right\rceil-1.$
Li and Sun \cite{LS23}  proved that
$\alpha_3(F)\geq \frac{2n}{3}$ for each
acyclic graph $F$ with order $n$. They also characterized all the corresponding
extremal acyclic graphs.
Bock et al. \cite{BP23a} generalized the result and  showed that if $G$ is a graph with
$n$ vertices, $m$ edges, $k$ components, and $c_1$ induced cycles of length 1 modulo 3, then
$\alpha_3(G)\geq n-\frac{1}{3}(m+k+c_1)$, the extremal graphs in which every two cycles are
vertex-disjoint were identified. In another paper, they \cite{BP23b}  provided several upper bounds on the dissociation number by utilizing  independence number in some specific classes of graphs, including bipartite graphs,  triangle-free graphs and subcubic graphs. The extremal graphs that reach the partial bounds were
characterized.

Since generalized $4$-independence  number is   a natural generalization of the independence number and the dissociation number, there has been a growing interest in the study of the generalized $4$-independence number.
Inspired by the work of \cite{TZS}, Li and Xu \cite{LX24} determined all the trees having the maximum number of maximum generalized $4$-independent sets among trees with given order. Li and Zhou \cite{LZ25} established a sharp lower bound on the generalized $4$-independence number of a tree with fixed order and characterized all the corresponding extremal trees.
In this paper, we present a sharp lower bound for the generalized $4$-independence number of a general graph using a novel approach. Additionally, we fully characterize the extremal graphs for which the generalized $4$-independence number reaches this lower bound.

\subsection{\normalsize Main results}

In this subsection, we give some basic notation and then describe our main result. For each positive integer $i,$ Li and Zhou \cite{LZ25} constructed a sequence of trees $R_i$ with order $4i$ as follows:
\begin{wst}
\item[{\rm (i)}] $R_1\in \{P_4,S_4\};$
\item[{\rm (ii)}] If $i\geq 2,$ then $R_{i}$ is obtained by adding an edge to connect a vertex of one member of  $R_{i-1}$ and a vertex of $P_4$ or $S_4.$
\end{wst}

 Li and Zhou established a lower bound on the generalized $4$-independence number of a tree with fixed order, and all the corresponding extremal trees were characterized, as stated below.

\begin{thm}[\cite{LZ25}]\label{thm1.1}
Let $T$ be a tree on $n$ vertices. Then $\alpha_4(T)\geq \frac{3n}{4}$ with equality if and only if $n\equiv 0\pmod{4}$ and $T\in R_{\frac{n}{4}}.$
\end{thm}

If $G$ is an $n$-vertex disconnected  graph with
$G=\bigcup_{i=1}^kG_i$, where $G_i$ is a  component of $G$ with order $n_i\ (1\leq i\leq k)$, then it is obvious that
$$
\sum_{i=1}^kn_i=n,\ \ \sum_{i=1}^k\alpha_4(G_i)=\alpha_4(G),
$$
and thus we arrive at the following result immediately.

\begin{thm}\label{thm1.2}
Let $F$ be  an acyclic graph with $n$ vertices. Then $\alpha_4(F)\geq \frac{3n}{4}$ with equality if and only if each component, say $T,$ of $F$ satisfies $|V_T|\equiv 0\pmod{4}$ and $T\in R_{\frac{|V_T|}{4}}.$
\end{thm}

Let $G$ be a graph with pairwise vertex-disjoint cycles, and let $\mathscr{C}_G$ denote the set of all cycles in $G$. By shrinking each cycle of $G$ (that is, contracting each cycle to a single vertex) we obtain an acyclic graph $T_G$ from $G$. More definitely, the vertex set $V_{T_G}=U_G\bigcup W_{\mathscr{C}_G},$ where $U_G$ consists of all vertices of $G$ that do not lie on any cycle and $W_{\mathscr{C}_G}$ consists of all the vertices each of which is obtained by shrinking a cycle in $\mathscr{C}_G$. Two vertices in $U_G$ are adjacent in $T_G$ if and only if they are adjacent in $G$, a vertex $u\in U_G$ is adjacent to a vertex $v_C\in W_{\mathscr{C}_G}$ if and only if $u$ is adjacent (in $G$) to a vertex in the cycle $C$, and two vertices $v_{C^1}, v_{C^2}$ in $W_{\mathscr{C}_G}$ are adjacent in $T_G$ if and only if there exists an edge in $G$ joining a vertex of $C^1\in \mathscr{C}_G$ to a vertex of $C^2\in \mathscr{C}_G.$  Observe that the graph $T_G-W_{\mathscr{C}_G}$ is the same as the graph obtained from $G$ by deleting all the vertices on cycles and their incident edges, the resultant graph is denoted by $\Gamma_G.$ Figure \ref{fig1} gives an example for $G, T_G$ and $\Gamma_G.$

Our main result extends Theorem \ref{thm1.2} to general graphs, which establishes a sharp lower bound on the generalized 4-independence number and characterizes all the corresponding extremal graphs.
\begin{thm}\label{thm1.4}
Let $G$ be an $n$-vertex  graph with the dimension of cycle space $\omega(G)$. Then
\begin{equation}\label{eq1.1}
\alpha_4(G)\geq \frac{3}{4}[n-\omega(G)]
\end{equation}
with equality if and only if all the following conditions hold for $G$
\begin{wst}
\item[{\rm (i)}] the cycles (if any) of $G$ are pairwise vertex-disjoint;
\item[{\rm (ii)}] the order of each cycle (if any) of $G$ is
$4k+1,$ where $k$ is an integer;
\item[{\rm (iii)}] each component, say $T,$  of $\Gamma_G$ satisfies $|V_T|\equiv 0\pmod{4}$ and $T\in R_{\frac{|V_T|}{4}}.$
\end{wst}
\end{thm}

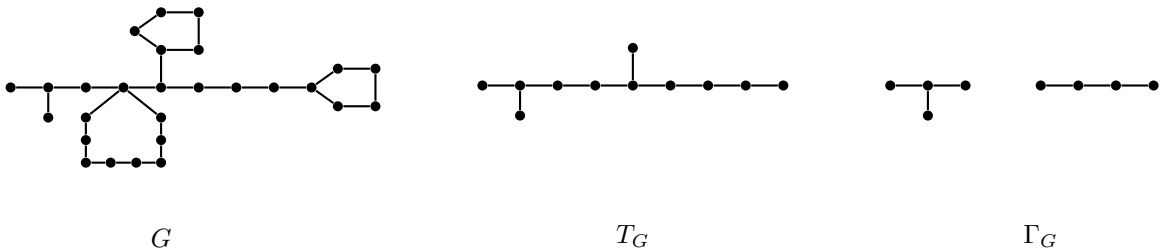
\begin{figure}[!ht]
\centering
  \begin{tikzpicture}[scale = 1.0]
  \tikzstyle{vertex}=[circle,fill=black,minimum size=0.38em,inner sep=0pt]
  \node[vertex] (G_0) at (0,0){};
  \node[vertex] (G_1) at (0.5,0){};
  \node[vertex] (G_2) at (1,0){};
  \node[vertex] (G_31) at (1.5,0){};
  \node[vertex] (G_6) at (2,0){};
  \node[vertex] (G_7) at (2.35,0.25){};
  \node[vertex] (G_9) at (2.35,-0.25){};
   \node[vertex] (G_26) at (2.85,0.25){};
  \node[vertex] (G_27) at (2.85,-0.25){};
  \node[vertex] (G_10) at (0,0.5){};
  \node[vertex] (G_11) at (-0.35,0.75){};
  \node[vertex] (G_12) at (0,1){};
  \node[vertex] (G_28) at (0.5,1){};
  \node[vertex] (G_29) at (0.5,0.5){};
  \node[vertex] (G_14) at (-0.5,0){};
  \node[vertex] (G_15) at (-1,0){};
  \node[vertex] (G_16) at (-1.5,0){};
  \node[vertex] (G_30) at (-1.5,-0.4){};
  \node[vertex] (G_17) at (-2,0){};
  \node[vertex] (G_18) at (-0.67,-1){};
  \node[vertex] (G_19) at (-0.33,-1){};
  \node[vertex] (G_20) at (-1,-0.4){};
  \node[vertex] (G_21) at (-1,-0.7){};
  \node[vertex] (G_22) at (-1,-1){};
  \node[vertex] (G_23) at (0,-1){};
  \node[vertex] (G_24) at (0,-0.7){};
  \node[vertex] (G_25) at (0,-0.4){};
  \draw[thick] (G_17)--(G_16) -- (G_15)--(G_14)--(G_0)--(G_1)--(G_2)--(G_31)--(G_6)
  --(G_7)--(G_26)--(G_27)--(G_9)--(G_6);
  \draw[thick] (G_0)--(G_10) -- (G_11)--(G_12)--(G_28)--(G_29)--(G_10);
  \draw[thick] (G_14)--(G_20) -- (G_21)--(G_22)--(G_18)--(G_19)--(G_23)
  --(G_24)--(G_25)--(G_14);
  \draw[thick] (G_16)--(G_30);
  \draw (0,-2.0)node{$G$};
  \end{tikzpicture}
  \hspace{3em}
  \begin{tikzpicture}[scale = 1.0]
  \tikzstyle{vertex}=[circle,fill=black,minimum size=0.38em,inner sep=0pt]
  \node[vertex] (G_0) at (0,0){};
  \node[vertex] (G_1) at (0.5,0){};
  \node[vertex] (G_2) at (1,0){};
  \node[vertex] (G_31) at (1.5,0){};
  \node[vertex] (G_6) at (2,0){};
  \node[vertex] (G_10) at (0,0.5){};
  \node[vertex] (G_14) at (-0.5,0){};
  \node[vertex] (G_15) at (-1,0){};
  \node[vertex] (G_16) at (-1.5,0){};
  \node[vertex] (G_30) at (-1.5,-0.4){};
  \node[vertex] (G_17) at (-2,0){};
  \draw[thick] (G_17)--(G_16) -- (G_15)--(G_14)--(G_0)--(G_1)--(G_2)--(G_31)--(G_6);
  \draw[thick] (G_0)--(G_10);
  \draw[thick] (G_16)--(G_30);
  \draw (0,-2.0)node{$T_G$};
  \end{tikzpicture}
  \hspace{3em}
  \begin{tikzpicture}[scale = 1.0]
  \tikzstyle{vertex}=[circle,fill=black,minimum size=0.38em,inner sep=0pt]
  \node[vertex] (G_0) at (0,0){};
  \node[vertex] (G_1) at (0.5,0){};
  \node[vertex] (G_2) at (1,0){};
  \node[vertex] (G_31) at (1.5,0){};
  \node[vertex] (G_15) at (-1,0){};
  \node[vertex] (G_16) at (-1.5,0){};
  \node[vertex] (G_30) at (-1.5,-0.4){};
  \node[vertex] (G_17) at (-2,0){};
  \draw[thick] (G_17)--(G_16) -- (G_15);
  \draw[thick] (G_0)--(G_1)--(G_2)--(G_31);
  \draw[thick] (G_16)--(G_30);
  \draw (0,-2.0)node{$\Gamma_G$};
  \end{tikzpicture}
  \caption{Graphs $G, T_G$ and $\Gamma_G$}\label{fig1}
\end{figure}

For example, if $G$ is the graph as depicted in Figure \ref{fig1}, then $G$ satisfies Theorem \ref{thm1.4} (i)-(iii) and $\alpha_4(G)=\frac{3}{4}[n-\omega(G)]$ holds, where $\alpha_4(G)=18, n=27$ and $\omega(G)=3$.

\section{\normalsize Preliminary results}\label{sec2}
In this section, we present some initial findings that will serve as the basis for proving our main result.
 The following  results
immediately follow from the definition of the  generalized $4$-independence number.
\begin{lem}\label{lem2.1}
Let $G=(V_G,E_G)$ be a simple graph. Then
\begin{wst}
\item[{\rm (i)}] $\alpha_4(G)-1\leq\alpha_4(G-v)\leq\alpha_4(G)$ for any $v\in V_G;$
\item[{\rm (ii)}] $\alpha_4(G-e)\geq \alpha_4(G)$ for any $e\in E_G$.
\end{wst}
\end{lem}

\begin{lem}\label{lem2.2}
Let $P_n$ and $C_n$ be a path and a cycle on $n$ vertices, respectively. Then
$\alpha_4(P_n)=\left\lceil\frac{3}{4}n\right\rceil$ and 
$\alpha_4(C_n)=\left\lfloor\frac{3}{4}n\right\rfloor.$
\end{lem}

Note that  the dimension of the cycle space of $G$  is actually the number of the
fundamental cycle in $G$.
The following result is obvious.
\begin{lem}\label{lem2.3}
Let $G$ be a graph with $x\in V_G$.
\begin{wst}
\item[{\rm (i)}] If $x$ lies outside any cycle of $G,$ then $\omega(G)=\omega(G-x);$
\item[{\rm (ii)}] If $x$ lies on a cycle, then $\omega(G-x)\leq \omega(G)-1;$
\item[{\rm (iii)}] If the cycles of $G$ are pairwise vertex-disjoint, then $\omega(G)$ precisely equals the number of cycles in $G$.
\end{wst}
\end{lem}

For a graph $G$, denote by $\mathcal{P}(G)$\ (resp. $\mathcal{Q}(G))$ the set of all pendant vertices (resp. quasi-pendant vertices) of $G$. In particular, denote by $\mathcal{Q}_2(G)=\{v\in \mathcal{Q}(G): d_G(v)=2\}$
and $\mathcal{Q}_3(G)=\{v\in \mathcal{Q}(G): d_G(v)=3, |N_G(v)\cap \mathcal{P}(G)|=2\}.$ In addition, let
$\mathcal{Q}_2'(G)=\{v\in V_G: d_G(v)=2, |N_G(v)\cap \mathcal{Q}_2(G)|=1\}.$

Li and Zhou \cite{LZ25} proved that if $G$ is a graph with at least 7 vertices, then there exists a maximum generalized 4-independent set in $G$ that can include all vertices in $\mathcal{P}(G)\cup \mathcal{Q}_2(G)\cup\mathcal{Q}_3(G)\cup\mathcal{Q}'_2(G).$ Their proof demonstrates that, under no order constraints, every connected graph admits a maximum generalized $4$-independent  set containing all pendant vertices. For the sake of completeness, we will also include the proof process below.

\begin{lem}\label{lem2.4}
Let $G$ be a simple connected graph. Then there exists a maximum generalized $4$-independent set in $G$ that can contain all vertices in $\mathcal{P}(G)$.
\end{lem}
\begin{proof}
 Let $S$ be a maximum generalized $4$-independent set that
 maximizes $|\mathcal{P}(G) \cap S|$. Suppose for contradiction
 there exists $u \in \mathcal{P}(G) \setminus S$, and let $v \in N_G(u)$.
 Necessarily $v \in S \setminus \mathcal{P}(G)$; otherwise $S\cup \{u\}$
 would be a generalized $4$-independent set, contradicting the maximality of
  $S$. The set $S':= (S \setminus \{v\}) \cup \{u\}$ is then a maximum
  generalized $4$-independent set with $|S' \cap \mathcal{P}(G)| > |S \cap \mathcal{P}(G)|$,
  violating the choice of $S$. Thus $\mathcal{P}(G) \subseteq S$.

\end{proof}

\section{\normalsize Proof of  Theorem \ref{thm1.4}}\setcounter{equation}{0}\label{sec3}
In this section, we give  a proof for Theorem \ref{thm1.4}, which
establishes a sharp lower bound on the generalized 4-independence number of a general graph. The corresponding extremal graphs are also characterized. More specifically, we give the proof  according to the following steps. We first show that the inequality in \eqref{eq1.1} holds. Then we present a few technical lemmas. Finally, we characterize all the graphs which attain the equality in \eqref{eq1.1}.
Without loss of generality,  assume that $G$ is  connected.
\vspace{2mm}

\begin{lem}\label{lem3.1}
The inequality \eqref{eq1.1} holds.
\end{lem}

\begin{proof}
We show \eqref{eq1.1} holds by induction on $\omega(G).$  If $\omega(G)=0$, then $G$ is a tree and the result follows immediately by Theorem \ref{thm1.1}.  Now assume that $\omega(G) \geq 1,$  i.e., $G$ has at least one cycle. Let $x$ be a vertex lying on some cycle.  By Lemmas \ref{lem2.1} and \ref{lem2.3}, we have
\begin{eqnarray}\label{eq3.1}
\alpha_4(G)\geq \alpha_4(G-x),\ \ \omega(G)\geq \omega(G-x)+1.
\end{eqnarray}
Applying the induction hypothesis yielding
\begin{eqnarray}\label{eq3.2}
\alpha_4(G-x)\geq\frac{3}{4}[n-1-\omega(G-x)].
\end{eqnarray}
Therefore,  \eqref{eq3.1}-\eqref{eq3.2} lead to
\begin{eqnarray*}
\alpha_4(G)\geq \frac{3}{4}[n-\omega(G)],
\end{eqnarray*}
as desired.
\end{proof}

For convenience, a graph is called \textit{good} if it achieves equality in \eqref{eq1.1}. In the following, we aim to provide some fundamental characterizations of good graphs. The following result is a direct consequence of Lemma \ref{lem3.1}.

\begin{lem}\label{lem3.2}
A disconnected graph is good if and only if each component of it is good.
\end{lem}

\begin{lem}\label{lem3.3}
Let $G$ be a graph with $x\in V_G$  lying on some cycle of $G$. If $G$ is good, then
\begin{wst}
\item[{\rm (i)}] $\alpha_4(G)=\alpha_4(G-x);$
\item[{\rm (ii)}] $\omega(G)=\omega(G-x)+1;$
\item[{\rm (iii)}] $G-x$ is good;
\item[{\rm (iv)}] $x\notin \mathcal{Q}(G)$ and $x$ is not adjacent to any vertex in $\mathcal{Q}_2(G)\cup\mathcal{Q}_3(G)\cup\mathcal{Q}'_2(G).$
\end{wst}
\end{lem}
\begin{proof}
The good condition for $G$ together with the proof of Lemma \ref{lem3.1} forces all equalities in \eqref{eq3.1}-\eqref{eq3.2}. Hence, (i)-(iii) are all derived.

If $x\in \mathcal{Q}(G),$ then $G-x$ contains an isolated vertex as its connected component. The fact that
an isolated vertex is not  good together with Lemma \ref{lem3.2} implies $G-x$ is not good, which contradicts to (iii). Therefore, $x\notin \mathcal{Q}(G).$

If $x$ is adjacent to some vertex in $\mathcal{Q}_2(G),$ then $P_2$ is a connected
component of $G-x$. By (iii) and Lemma \ref{lem3.2}, one has $P_2$ is good,  which is clearly impossible. Therefore, $x$ is not adjacent to any vertex in $\mathcal{Q}_2(G).$

If $x$ is adjacent to some vertex in $\mathcal{Q}_3(G)\cup\mathcal{Q}'_2(G),$ then $P_3$ is a connected
component of $G-x$. It follows from  (iii) and Lemma \ref{lem3.2} that $P_3$ is good,  which clearly leads to a contradiction. Consequently, $x$ is not adjacent to any vertex in $\mathcal{Q}_3(G)\cup\mathcal{Q}'_2(G).$
This completes the proof of (iv).

\end{proof}

An induced cycle $H$  of a graph $G$ is called a \textit{pendant cycle} if $H$ contains a unique vertex of degree 3 and each of its rest vertices is of degree 2 in $G$. For example, the graph $G$ as depicted in Figure \ref{fig1}  has exactly two pendant cycles
(the $9$-vertex cycle is not a pendant cycle).

\begin{lem}\label{lem3.4}
Let $G$ be a graph with $C_q$ being a pendant cycle of $G$. Denote by $H=G-C_q$. If $G$ is good, then
\begin{wst}
\item[{\rm (i)}] $q\equiv 1\pmod{4};$
\item[{\rm (ii)}] $\alpha_4(G)=\alpha_4(H)+\frac{3}{4}(q-1);$
\item[{\rm (iii)}] $H$ is good.
\end{wst}
\end{lem}
\begin{proof}
Let $x$ be the unique vertex of degree $3$ on $C_q.$ Then $G-x=P_{q-1}\cup H$.
By Lemmas \ref{lem3.2} and \ref{lem3.3}(iii), we obtain both $P_{q-1}$ and $H$ are good. This means $\alpha_4(P_{q-1})=\frac{3}{4}(q-1)$ and $q\equiv 1\pmod{4}$
by Theorem \ref{thm1.1} and Lemma \ref{lem2.2}.
 Applying
Lemma \ref{lem3.3}(i), one has
$$
\alpha_4(G)=\alpha_4(G-x)=\alpha_4(H)+\alpha_4(P_{q-1})
=\alpha_4(H)+\frac{3}{4}(q-1).
$$
This completes the proof.
\end{proof}

Specializing a vertex in a tree yields a so-called \textit{rooted tree},
where the specialized vertex is called the \textit{root} of this tree. In a rooted
tree $T$, the length of the unique path $rTv$ from the root $r$ to the vertex
$v$ is called the \textit{level} of $v$, denoted by $l(v)$. Each vertex on
the path $rTv$, not including the vertex $v$ itself, is called an \textit{ancestor} of
$v$, and each vertex with $v$ as its ancestor is a \textit{descendant} of $v$.
The immediate ancestor of $v$ is its \textit{parent}, and the vertices whose parent is
$v$ are its \textit{children}. Denote by $O_{\cup_{k\geq2}(b_kS_k)\cup bP_4}$ the graph obtained
by identifying one  leaf from each of \( b \) copies of \( P_4 \), \( b_2 \) copies of \( S_2 \), \( b_i\) copies of \( S_i\ (i\geq 3) \), see Figure \ref{fig04} for an example of $O_{\cup_{k=2}^6(b_kS_k)\cup bP_4}.$

\begin{figure}[!ht]
\centering
  \begin{tikzpicture}[scale = 1.1]
  \tikzstyle{vertex}=[circle,fill=black,minimum size=0.38em,inner sep=0pt]
  \node[vertex] (1) at (0,0){};
  \node[vertex] (2) at (-2.5,1){};
  \node[vertex] (3) at (-2.5,1.5){};
  \node[vertex] (4) at (-2.5,2){};
  \node[vertex] (5) at (-1.5,1){};
  \node[vertex] (6) at (-1.5,1.5){};
  \node[vertex] (7) at (-1.5,2){};
  \node[vertex] (8) at (-0.5,1){};
  \node[vertex] (9) at (0.5,1){};
  \node[vertex] (10) at (1.5,1){};
  \node[vertex] (11) at (1.5,1.5){};
  \node[vertex] (12) at (2.5,1){};
  \node[vertex] (13) at (2.5,1.5){};
  \node[vertex] (14) at (-2.5,-1){};
  \node[vertex] (15) at (-2.8,-1.5){};
  \node[vertex] (16) at (-2.2,-1.5){};
  \node[vertex] (17) at (-1.5,-1){};
  \node[vertex] (18) at (-1.8,-1.5){};
  \node[vertex] (19) at (-1.2,-1.5){};
  \node[vertex] (20) at (-0.5,-1){};
  \node[vertex] (21) at (-0.8,-1.5){};
  \node[vertex] (22) at (-0.5,-1.5){};
  \node[vertex] (23) at (-0.2,-1.5){};
  \node[vertex] (24) at (0.5,-1){};
  \node[vertex] (25) at (0.2,-1.5){};
  \node[vertex] (26) at (0.5,-1.5){};
  \node[vertex] (27) at (0.8,-1.5){};
  \node[vertex] (28) at (1.5,-1){};
  \node[vertex] (29) at (1.2,-1.5){};
  \node[vertex] (30) at (1.4,-1.5){};
  \node[vertex] (31) at (1.6,-1.5){};
  \node[vertex] (32) at (1.8,-1.5){};
  \node[vertex] (33) at (2.5,-1){};
  \node[vertex] (34) at (2.2,-1.5){};
  \node[vertex] (35) at (2.4,-1.5){};
  \node[vertex] (36) at (2.6,-1.5){};
  \node[vertex] (37) at (2.8,-1.5){};
  \draw[thick] (1) -- (2)--(3)--(4);
  \draw[thick] (1) -- (5)--(6)--(7);
  \draw[thick] (1)--(8);
  \draw[thick] (1)--(9);
  \draw[thick] (1) -- (10)--(11);
  \draw[thick] (1) -- (12)--(13);
  \draw[thick] (1)--(14)--(15);
  \draw[thick] (14)--(16);
  \draw[thick] (1)--(17)--(18);
  \draw[thick] (17)--(19);
  \draw[thick] (1)--(20)--(21);
  \draw[thick] (20)--(22);
  \draw[thick] (20)--(23);
  \draw[thick] (1)--(24)--(25);
  \draw[thick] (24)--(26);
  \draw[thick] (24)--(27);
  \draw[thick] (1)--(28)--(29);
  \draw[thick] (28)--(30);
  \draw[thick] (28)--(31);
  \draw[thick] (28)--(32);
  \draw[thick] (1)--(33)--(34);
  \draw[thick] (33)--(35);
  \draw[thick] (33)--(36);
  \draw[thick] (33)--(37);
  \draw (0,1)node{$\cdots$};
  \draw (-2,1)node{$\cdots$};
  \draw (2,1)node{$\cdots$};
  \draw (0,-1)node{$\cdots$};
  \draw (-2,-1)node{$\cdots$};
  \draw (2,-1)node{$\cdots$};
  \draw (0,1.4)node{$b_2$};
  \draw (-2,1.4)node{$b$};
  \draw (2,1.4)node{$b_3$};
  \draw (0,-1.3)node{$b_5$};
  \draw (-2,-1.3)node{$b_4$};
  \draw (2,-1.3)node{$b_6$};
  \draw (0.17,1.15)node{\scalebox{3.0}[1.5]{\rotatebox{90}{\}} }};
  \draw (-1.83,1.15)node{\scalebox{3.0}[1.5]{\rotatebox{90}{\}} }};
  \draw (2.17,1.15)node{\scalebox{3.0}[1.5]{\rotatebox{90}{\}} }};
  \draw (0.14,-1.12)node{\scalebox{2.8}[1.5]{\rotatebox{90}{\{} }};
  \draw (-1.86,-1.12)node{\scalebox{2.8}[1.5]{\rotatebox{90}{\{} }};
  \draw (2.14,-1.12)node{\scalebox{2.8}[1.5]{\rotatebox{90}{\{} }};
  \end{tikzpicture}
  \caption{Graph $O_{\cup_{k=2}^6(b_kS_k)\cup bP_4}.$}\label{fig04}
\end{figure}
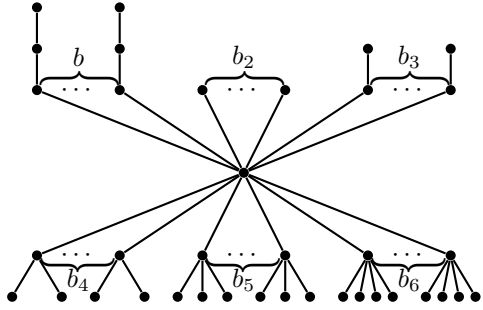

\begin{lem}\label{lem3.6}
If $G$ is good, then
\begin{wst}
\item[{\rm (i)}] the cycles (if any) of $G$ are pairwise vertex-disjoint;
\item[{\rm (ii)}] the order of each cycle (if any) of $G$ is
$4k+1,$ where $k$ is a positive integer;
\item[{\rm (iii)}] $\alpha_4(G)=\alpha_4(\Gamma_G)+\sum_{C\in \mathscr{C}_G}\frac{3}{4}(|V_C|-1).$
\end{wst}
\end{lem}
\begin{proof}
If $G$ is not cycle-disjoint, then there exists a vertex, say \(x\), on some cycle of \(G\) such that $c(G-x)\leq d_G(x)-2.$
This together with the fact that $|V_{G-x}|=|V_G|-1$ and $|E_{G-x}|=|E_G|-d_G(x)$ yields $\omega(G-x)\leq \omega(G)-2,$ which contradicts to Lemma \ref{lem3.3}(ii).
This completes the proof of (i).

We proceed by induction on the order $n$ of $G$ to prove (ii) and (iii). Since  $P_1,P_2,P_3$ and $C_3$ are not good, we have $n\geq4.$
If $n=4,$ then
$G\in\{ P_4,S_4, C_3^+,C_4,K_4^-,K_4\},$ where $C_3^+$ is the graph
obtained from $C_3$ by adding a pendent edge at any vertex and $K_4^-$
is the graph obtained from $K_4$ by deleting an edge. Through individual verification,
it can be concluded that only $P_4$ and $S_4$ are  good  among $\{ P_4,S_4, C_3^+,C_4,K_4^-,K_4\},$
 and  $G$ is ultimately isomorphic to $P_4$ or $S_4$. Thus, (ii)-(iii) establish obviously.
Suppose that (ii) and (iii) hold for any good graph of order smaller than $n$, and suppose $G$ is a good graph with order $n\geq 5.$

If $T_G$ is an empty graph, then $G\cong C_n.$ Thus (ii) and (iii)  follow from  the following two facts.
\begin{fact}\label{fact1}
$C_n$ is good if and only if $n\equiv 1\pmod{4}.$
\end{fact}
\begin{fact}\label{fact2}
$\alpha_4(C_n)=\frac{3}{4}(n-1)$ if $n\equiv 1\pmod{4}$.
\end{fact}

If $T_G$ contains at least one edge, then $\mathcal{P}(T_G)\neq\emptyset$. In order to complete the proof of (ii) and (iii) in this case, it suffices to consider the following two possible cases.

{\bf{Case 1.}}\ $\mathcal{P}(T_G)\cap  W_{\mathscr{C}_G}\neq\emptyset$. In this case, $G$ has a pendant cycle, say $C_q$. Let $H=G-C_q$. It follows from Lemma \ref{lem3.4}(iii) that $H$ is good. Applying the induction hypothesis to $H$ yields
\begin{wst}
\item[{\rm (a)}] the order of each cycle (if any) of $H$ is
$4k+1,$ where $k$ is a positive integer;
\item[{\rm (b)}] $\alpha_4(H)=\alpha_4(\Gamma_H)+\sum_{C\in \mathscr{C}_H}\frac{3}{4}(|V_C|-1).$
\end{wst}
Assertion (a) and Lemma \ref{lem3.4}(i) imply that the order of each cycle of $G$ is  1 modulo 4 since $\mathscr{C}_G=\mathscr{C}_{H}\cup\{C_q\}.$ Thus, (ii) holds in this case.

Combining with Lemma \ref{lem3.4}(ii) and Assertion (b) we have
\begin{eqnarray}\label{eq3.5}
\alpha_4(G)=\alpha_4(H)+\frac{3}{4}(q-1)=\alpha_4(\Gamma_H)
+\sum_{C\in \mathscr{C}_H}\frac{3}{4}(|V_C|-1)+\frac{3}{4}(q-1).
\end{eqnarray}
Note that $\Gamma_G\cong \Gamma_H$ and
$$
\sum_{C\in \mathscr{C}_H}\frac{3}{4}(|V_C|-1)+\frac{3}{4}(q-1)=\sum_{C\in \mathscr{C}_G}\frac{3}{4}(|V_C|-1).
$$
Together with (\ref{eq3.5}) we have
\begin{eqnarray*}
\alpha_4(G)=\alpha_4(\Gamma_G)+\sum_{C\in \mathscr{C}_G}\frac{3}{4}(|V_C|-1).
\end{eqnarray*}
That is to say, (iii) holds in this case.

\vspace{2mm}

{\bf{Case 2.}}\ $\mathcal{P}(T_G)\cap  W_{\mathscr{C}_G}=\emptyset$.
In this case, $\mathcal{P}(T_G)=\mathcal{P}(G)$.  Assume
$u\in V_{T_G}$ such that all of its children are leaves, and subject to this condition,
the level of $u$ is as large as possible. If $u\in W_{\mathscr{C}_G}, $ then there exists
a quasi-pendant vertex on some cycle of $G$, which is impossible by Lemma \ref{lem3.3}(iv).
Hence, $u\notin W_{\mathscr{C}_G}.$ If $l_{T_G}(u)=0,$ then $T_G\cong G\cong S_n$ and thus $\alpha_4(G)=n-1>\frac{3}{4}n$
for $n\geq5,$ contradicts to the
fact that $G$ is good. Therefore, $l_{T_G}(u)\geq1$.

\vspace{1mm}
{\bf{Subcase 2.1}}\ $d_G(u)=d_{T_G}(u)\geq4,$  then $u$ has at least three children in $T_G$.
Recall that $u\notin W_{\mathscr{C}_G}$, we get $V_{T_G^u}\cap W_{\mathscr{C}_G}=\emptyset$,
where $T_G^u$ is a subtree of $T_G$ rooted at $u$.
Let $H_1$ be the subgraph of $G$ such that
$T_{H_1}=T_G-T_G^u$ and $\mathscr{C}_{H_1}=\mathscr{C}_G.$
  Then we arrive at
\begin{eqnarray}\label{eq3.4a}
|V_{H_1}|=n-d_G(u),\ \  \omega(H_1)=\omega(G).
\end{eqnarray}
Since all children of $u$ are in $\mathcal{P}(G),$ there exists a maximum generalized $4$-independent set $S(G)$ of $G$ such that all children of $u$ are in it
by Lemma \ref{lem2.4}.   Thus,  $u\notin S(G).$ Therefore, we get
\begin{eqnarray}\label{eq3.5a}
\alpha_4(H_1)=\alpha_4(G)-d_G(u)+1.
 \end{eqnarray}
Equalities \eqref{eq3.4a}-\eqref{eq3.5a} together with the
good condition of $G$ lead to
$$
\alpha_4(H_1)=\frac{3}{4}(|V_{H_1}|-\omega(H_1))-\frac{1}{4}(d_G(u)-4).
$$
In view of Lemma \ref{lem3.1}, we get $d_G(u)=4$ and $H_1$ is good. Applying the induction hypothesis to $H_1$ implies that
\begin{wst}
\item[{\rm (c)}] the order of each cycle (if any) of $H_1$
is $4k+1,$ where $k$ is an integer;
\item[{\rm (d)}] $\alpha_4(H_1)=\alpha_4(\Gamma_{H_1})+\sum_{C\in \mathscr{C}_{H_1}}\frac{3}{4}(|V_C|-1).$
\end{wst}

Since $\mathscr{C}_G= \mathscr{C}_{H_1}$,  combining with Assertion (c) we have the order of each cycle (if any) of $G$ is  $4k+1$, where $k$ is an integer.

Note that $V_{T_G^u}\subseteq\{u\}\cup\mathcal{P}(\Gamma_G)$  and $\Gamma_{H_1}=\Gamma_G-V_{T_G^u}$. Hence, by applying Equality \eqref{eq3.5a}, Lemma \ref{lem2.4} and Assertion (d) we get
$$
\alpha_4(G)=\alpha_4(H_1)+3=\alpha_4(\Gamma_{H_1})+\sum_{C\in \mathscr{C}_{H_1}}\frac{3}{4}(|V_C|-1)+3=\alpha_4(\Gamma_G)+\sum_{C\in \mathscr{C}_G}\frac{3}{4}(|V_C|-1).
$$

\vspace{1mm}
{\bf{Subcase 2.2}}\ $d_G(u)=d_{T_G}(u)=3,$ i.e., $u\in \mathcal{Q}_3(G)$. Let $w$ be the parent of $u$ in $T_G$. In view of the proof as above, it is sufficient to consider that each child of $w$ is either a
pendant vertex or a quasi-pendant vertex of degree two or three in $T_G$. Put $A_i:=\{v\in N_{T_G^w}(w)\ \hbox{ and }\ d_{T_G^w}(v)=i\}$ and $a_i:=|A_i|$ for $i\in\{1,2,3\}.$
If $w\in W_{\mathscr{C}_G},$ then there exists a  vertex adjacent to $u\in \mathcal{Q}_3(G)$ on some cycle of $G$, which contradicts to  Lemma \ref{lem3.3}(iv). Therefore,  $w\notin W_{\mathscr{C}_G}$.
Since  $V_{T_G^w}\subseteq\{w\}\cup\mathcal{P}(T_G)\cup\mathcal{Q}(T_G),$ again by Lemma \ref{lem3.3}, $V_{T_G^w}\cap W_{\mathscr{C}_G}=\emptyset$.
Let $H_2$ be the subgraph  of $G$ such that $T_{H_2}=T_G-T_G^w$ and $\mathscr{C}_{H_2}=\mathscr{C}_G$.  Then

\begin{eqnarray}\label{eq3.4b}
|V_{H_2}|=n-a_1-2a_2-3a_3-1,\ \ \omega(H_2)=\omega(G).
\end{eqnarray}
In view of  Lemma \ref{lem2.4}, there exists a maximum generalized $4$-independent set $S(G)$
of $G$ such that $\mathcal{P}(G)\subseteq S(G).$ If $w\notin S(G),$ then $V_{T_G^w}\setminus \{w\}\subseteq S(G)$ and thus $S(G)\setminus\{V_{T_G^w}\setminus \{w\}\}$ forms a maximum generalized $4$-independent set
of $H_2$. If $w\in S(G),$ then $S(G)\cap(A_2\cup A_3)=\emptyset.$ This together with $a_3\geq1$
means $\left(S(G)\setminus \{w\}\right)\cup A_2\cup A_3$ is a maximum generalized $4$-independent set of $G,$
and then $S(G)\setminus\{V_{T_G^w}\setminus \{w\}\}$ still forms a maximum generalized $4$-independent set
of $H_2$. Therefore, we  can conclude
\begin{eqnarray}\label{eq3.5b}
\alpha_4(H_2)=\alpha_4(G)-a_1-2a_2-3a_3.
 \end{eqnarray}
Recall that $G$ is good,  \eqref{eq3.4b}-\eqref{eq3.5b} yield
$$
\alpha_4(H_2)=\frac{3}{4}(|V_{H_2}|-\omega(H_2))-\frac{1}{4}(a_1+2a_2+3a_3-3).
$$
Since $u\in A_3,$ we have $a_3\geq1.$ It follows from
Lemma \ref{lem3.1} that  $a_1=a_2=0, a_3=1$ and $H_2$ is good.
By applying  the induction hypothesis to $H_2,$ it can be  concluded that
\begin{wst}
\item[{\rm (e)}] the order of each cycle (if any) of $H_2$
is $4k+1,$ where $k$ is an integer;
\item[{\rm (f)}] $\alpha_4(H_2)=\alpha_4(\Gamma_{H_2})+\sum_{C\in \mathscr{C}_{H_2}}\frac{3}{4}(|V_C|-1).$
\end{wst}

The fact that $\mathscr{C}_G= \mathscr{C}_{H_2}$ together with Assertion (e) implies the order of each cycle (if any) of $G$ is  $4k+1$, where $k$ is an integer.

Note that $V_{T_G^w}\subseteq\{w\}\cup\mathcal{P}(\Gamma_G)\cup\mathcal{Q}_3(\Gamma_G)$  and $\Gamma_{H_2}=\Gamma_G-V_{T_G^w}$. Combining Equality \eqref{eq3.5b}, Lemma \ref{lem2.4} and Assertion (f) we get
$$
\alpha_4(G)=\alpha_4(H_2)+3=\alpha_4(\Gamma_{H_2})+\sum_{C\in \mathscr{C}_{H_2}}\frac{3}{4}(|V_C|-1)+3=\alpha_4(\Gamma_G)+\sum_{C\in \mathscr{C}_G}\frac{3}{4}(|V_C|-1).
$$

\vspace{1mm}
{\bf{Subcase 2.3}}\ $d_G(u)=d_{T_G}(u)=2,$ i.e., $u\in \mathcal{Q}_2(G)$. Let $w$ be the parent of $u$ in $T_G$. In view of the proof as above, it is sufficient to consider that each child of $w$ is either a
pendant vertex or a quasi-pendant vertex of degree two in $T_G$. Let $A_i$ and $a_i$ be defined as before for $i\in\{1,2\}.$
If $w\in W_{\mathscr{C}_G},$ then there exists a  vertex adjacent to $u\in \mathcal{Q}_2(G)$ on some cycle of $G$, leading to a contradiction  by Lemma \ref{lem3.3}(iv). Whence,  $w\notin W_{\mathscr{C}_G}$.

If $d_G(w)\geq3,$ then $a_1+a_2\geq2.$ In a similar way as above, we have $V_{T_G^w}\cap W_{\mathscr{C}_G}=\emptyset$. Let $H_3$ be the subgraph of $G$ such that
$T_{H_3}=T_G-T_G^w$ and $\mathscr{C}_{H_3}=\mathscr{C}_G$. Then
\begin{eqnarray}\label{eq3.4c}
|V_{H_3}|=n-a_1-2a_2-1,\ \ \omega(H_3)=\omega(G).
\end{eqnarray}
Analogous to the discussion before \eqref{eq3.5b}, there always exists a maximum
generalized $4$-independent set $S(G)$
of $G$ such that $w\notin S(G)$ and then  $S(G)\setminus\{V_{T_G^w}\setminus \{w\}\}$  forms a maximum generalized $4$-independent set
of $H_3$. Hence, we obtain
\begin{eqnarray}\label{eq3.5c}
\alpha_4(H_3)=\alpha_4(G)-a_1-2a_2.
 \end{eqnarray}
Recall that $G$ is good, then it follows from \eqref{eq3.4c}-\eqref{eq3.5c} that
$$
\alpha_4(H_3)=\frac{3}{4}(|V_{H_3}|-\omega(H_3))-\frac{1}{4}(a_1+2a_2-3).
$$
The above equality together with Lemma \ref{lem3.1} and $u\in A_2$ forces $a_1=a_2=1$ and thus $H_3$ is good.
Applying the induction hypothesis to $H_3$ leads to
\begin{wst}
\item[{\rm (g)}] the order of each cycle (if any) of $H_3$
is $4k+1,$ where $k$ is an integer;
\item[{\rm (h)}] $\alpha_4(H_3)=\alpha_4(\Gamma_{H_3})+\sum_{C\in \mathscr{C}_{H_3}}\frac{3}{4}(|V_C|-1).$
\end{wst}

Combining with Assertion (g) and  $\mathscr{C}_G= \mathscr{C}_{H_3}$, we know the order of each cycle (if any) of $G$ is  $4k+1$, where $k$ is an integer.

It is routine to check that $V_{T_G^w}\subseteq\{w\}\cup\mathcal{P}(\Gamma_G)\cup\mathcal{Q}_2(\Gamma_G)$  and $\Gamma_{H_3}=\Gamma_G-V_{T_G^w}$. Then Equality \eqref{eq3.5c}, Lemma \ref{lem2.4} and Assertion (h) lead to
$$
\alpha_4(G)=\alpha_4(H_3)+3=\alpha_4(\Gamma_{H_3})+\sum_{C\in \mathscr{C}_{H_3}}\frac{3}{4}(|V_C|-1)+3=\alpha_4(\Gamma_G)+\sum_{C\in \mathscr{C}_G}\frac{3}{4}(|V_C|-1).
$$

If $d_G(w)=2,$ then let $s$ be the parent of $w$ in $T_G$. In view of the proof as above, it is sufficient to consider that $T_G^s\cong O_{\cup_{k\geq2}(b_kS_k)\cup bP_4}.$ In a similar way, we have $V_{T_G^s}\cap W_{\mathscr{C}_G}=\emptyset.$ Let $H_4$ be the subgraph of $G$ such that $T_{H_4}=T_G-T_G^s$
and $\mathscr{C}_{H_4}=\mathscr{C}_G$. Then
\begin{eqnarray}\label{eq3.4d}
|V_{H_4}|=n-\sum_{k\geq2}(k-1)b_k-3b-1,\ \ \omega(H_4)=\omega(G).
\end{eqnarray}
Again, through discussion similar to that preceding \eqref{eq3.5b}, we know
\begin{eqnarray}\label{eq3.5d}
\alpha_4(H_4)=\alpha_4(G)-b_2-2b_3-3b_4-\sum_{k\geq5}(k-2)b_k-3b.
 \end{eqnarray}
Equalities \eqref{eq3.4d}-\eqref{eq3.5d} together with the fact that $G$ is good yield
$$
\alpha_4(H_4)=\frac{3}{4}\left(|V_{H_4}|-\omega(H_4)\right)-\frac{1}{4}\left(b_2+2b_3+3b_4+\sum_{k\geq5}(k-5)b_k+3b-3\right).
$$
Note that $b\geq1$, then by Lemma \ref{lem3.1}, one has $b=1$ and $ b_k=0$ for $k\geq2$ and $k\neq 5.$  Thus,  $H_4$ is good.
Applying the induction hypothesis to $H_4$ implies that
\begin{wst}
\item[{\rm (i)}] the order of each cycle (if any) of $H_4$
is $4k+1,$ where $k$ is an integer;
\item[{\rm (j)}] $\alpha_4(H_4)=\alpha_4(\Gamma_{H_4})+\sum_{C\in \mathscr{C}_{H_4}}\frac{3}{4}(|V_C|-1).$
\end{wst}
Assertion (i) and  $\mathscr{C}_G= \mathscr{C}_{H_4}$ lead to  the order of each cycle (if any) of $G$ is  $4k+1$, where $k$ is an integer.

Note that $\Gamma_{H_4}=\Gamma_G-V_{T_G^s}$. Then it follows from  \eqref{eq3.5d}, Lemma \ref{lem2.4} and Assertion (j) that
$$
\alpha_4(G)=\alpha_4(H_4)+3b_5+3=\alpha_4(\Gamma_{H_4})+\sum_{C\in \mathscr{C}_{H_4}}\frac{3}{4}(|V_C|-1)+3b_5+3=\alpha_4(\Gamma_G)+\sum_{C\in \mathscr{C}_G}\frac{3}{4}(|V_C|-1),
$$
as desired.
\end{proof}

With the help of the above lemmas, we are ready to prove Theorem \ref{thm1.4} as  follows. \vspace{2mm}

\noindent{\bf Proof of Theorem \ref{thm1.4}.}\ Inequality \eqref{eq1.1} has already established by
Lemma~\ref{lem3.1}. We now characterize all the graphs which attain the lower bound by considering the sufficient and necessary conditions for the equality in \eqref{eq1.1}.

For ``sufficiency", Assertion (i) and Lemma
\ref{lem2.3}(iii) imply that $G$ has exactly $\omega(G)$
cycles, that is,
\begin{eqnarray}\label{eq3.10}
|\mathscr{C}_G|=\omega(G).
\end{eqnarray}
Assertion (ii) together with Lemma \ref{lem2.2} yields
\begin{eqnarray}\label{eq3.11}
\alpha_4(C)=\frac{3}{4}(|V_C|-1)
\end{eqnarray}
for any $C\in \mathscr{C}_G.$
Combining Assertion (iii) and
Theorem \ref{thm1.2}, we have
\begin{eqnarray}\label{eq3.12}
\alpha_4(\Gamma_G)=\frac{3}{4}|V_{\Gamma_G}|.
\end{eqnarray}
Note that the graph
$\bigcup_{C\in \mathscr{C}_G}C \bigcup \Gamma_G$ can be obtained from
$G$ by removing some edges. By Lemma~\ref{lem2.1}(ii) and \eqref{eq3.10}-\eqref{eq3.12}, we get
\begin{eqnarray*}
\alpha_4(G)&\leq& \alpha_4(\Gamma_G)+\sum_{C\in \mathscr{C}_G}\alpha_4(C)\\
&=&\frac{3}{4}|V_{\Gamma_G}|+\sum_{C\in\mathscr{C}_G}\frac{3}{4}(|V_C|-1)\\
&=&\frac{3}{4}\left(|V_{\Gamma_G}|+\sum_{C\in\mathscr{C}_G}|V_C|\right)-\frac{3}{4}\omega(G)\\
&=&\frac{3}{4}[n-\omega(G)].
\end{eqnarray*}
Therefore, $\alpha_4(G)=\frac{3}{4}[n-\omega(G)]$ by Lemma \ref{lem3.1}.

\vspace{2mm}
For ``necessity", let $G$ be a good graph. By Lemma \ref{lem3.6}, the cycles (if any) of $G$ are pairwise vertex-disjoint, and the order of each  cycle (if any) of $G$ is 1 modulo 4. This implies (i) and (ii). The good condition of $G$ together with \eqref{eq3.10}-\eqref{eq3.11} yields
\begin{eqnarray*}
\alpha_4(G)&=&\frac{3}{4}(n-\omega(G))\\
&=&\frac{3}{4}\left(|V_{\Gamma_G}|+\sum_{C\in\mathscr{C}_G}|V_C|-|\mathscr{C}_G|\right)\\
&=&\frac{3}{4}|V_{\Gamma_G}|+\sum_{C\in\mathscr{C}_G}\frac{3}{4}(|V_C|-1).
\end{eqnarray*}
Applying Lemma \ref{lem3.6}(iii), we get
$\alpha_4(\Gamma_G)=\frac{3}{4}|V_{\Gamma_G}|,$ which implies
(iii) by Theorem \ref{thm1.2}. \qed

\section{\normalsize Conclusion}\setcounter{equation}{0}\label{sec4}

For a graph $G$ with $n$ vertices, $m$ edges, $k$ components, and $c_1$ induced cycles of length 1 modulo 3, Bock et al. \cite{BP23a} showed that $\alpha_3(G)\geq n-\frac{1}{3}(m+k+c_1)$, the extremal graphs in which every two cycles are vertex-disjoint were identified.
In this paper, for a general graph $G$
with $n$ vertices, it is   proved that $\alpha_4(G)\geq \frac{3}{4}(n-\omega(G))$, where $\omega(G)$ denotes the dimension of the cycle space of $G$.  The  graph $G$ whose generalized $4$-independence number attains the lower bound are characterized completely. It is natural and interesting to investigate the  sharp lower bound of the generalized $k$-independence number for general integer  $k\geq2.$
We finish this section by proposing the following conjecture.

For an integer $k\geq2,$ constructing a sequence of trees $R_i$ with order $ki$ as follows:
\begin{wst}
\item[{\rm (i)}] $R_1$ is a $k$-tree;
\item[{\rm (ii)}] If $i\geq 2,$ then $R_{i}$ is obtained by adding an edge to connect a vertex of one member of  $R_{i-1}$ and a vertex of  a $k$-tree.
\end{wst}

\begin{conj}\label{conj1}
Let $G$ be an $n$-vertex  simple graph with the dimension of cycle space $\omega(G)$. Then
\begin{equation*}
\alpha_k(G)\geq \frac{k-1}{k}[n-\omega(G)]
\end{equation*}
with equality if and only if all the following conditions hold for $G$
\begin{wst}
\item[{\rm (i)}] the cycles (if any) of $G$ are pairwise vertex-disjoint;
\item[{\rm (ii)}] the order of each cycle (if any) of $G$ is 1 modulo $k;$
\item[{\rm (iii)}] each component, say $T,$  of $\Gamma_G$ satisfies $|V_T|\equiv 0\pmod{k}$ and $T\in R_{\frac{|V_T|}{k}}.$
\end{wst}
\end{conj}

It is important to emphasize that our method becomes infeasible for Conjecture  \ref{conj1}  when $k$ is  sufficiently large. The primary difficulty arises in the proof of Lemma~3.5, specifically in Case~2. Specifically, we analyze three distinct cases based on the value of $d_G(u)$, namely, $d(u) \geq 4$, $d(u) = 3$, and $d(u) = 2$. While the structure of the rooted tree $T_G^w$ is well-defined in each of these individual cases, increasing values of $k$ induce a combinatorial proliferation of subcases. Consequently, the structural complexity of $T_G^w$ escalates beyond feasible characterization.

\section*{\normalsize Acknowledgment}

The authors would like to express their sincere gratitude to both of the referees
for a very careful reading of this paper and for all their insightful comments, which led to a number of improvements.


\begin{thebibliography}{99}

\bibitem{A07} V.E. Alekseev, R. Boliac, D.V. Korobitsyn, V.V. Lozin, NP-hard graph problems and boundary classes of graphs, Theoret. Comput. Sci. 389 (2007) 219-236.

    \bibitem{BP23a}  F. Bock, J. Pardey, L.D. Penso, D. Rautenbach, A bound on the dissociation number,
J. Graph Theory 103 (2023) 661-673.

\bibitem{BP23b}  F. Bock, J. Pardey, L.D. Penso, D. Rautenbach, Relating the independence number and the dissociation number, J. Graph Theory 104 (2023) 320-340.

\bibitem{B04} R. Boliac, K. Cameron, V.V. Lozin, On computing the dissociation number and the induced matching number of bipartite graphs, Ars Combin. 72 (2004) 241-253.


\bibitem{BJ13} B. Bre\v{s}ar, M. Jakovac, J. Katreni\v{c}, G. Semani\v{s}in, A. Taranenko, On the vertex $k$-path cover, Discrete Appl. Math. 161 (2013) 1943-1949.

\bibitem{B11} B. Bre\v{s}ar, F. Kardo\v{s}, J. Katreni\v{c}, G. Semani\v{s}in, Minimum $k$-path vertex cover, Discrete Appl. Math. 159 (2011) 1189-1195.


\bibitem{C06} K. Cameron, P. Hell, Independent packings in structured graphs, Math. Program. 105 (2006) 201-213.

\bibitem{C95} D. Cvetkovi\'{c}, M. Doob, H. Sachs, Spectra of Graphs, third ed., Johann Ambrosius Barth, Heidelberg, 1995.



\bibitem{GH09}  F. G\"{o}ring, J. Harant, D. Rautenbach, I. Schiermeyer, On $F$-independence in graphs, Discuss. Math. Graph Theory 29 (2009) 377-383.

\bibitem{HB22} S. Hosseinian, S. Butenko, An improved approximation for maximum $k$-dependent set on bipartite graphs, Discrete Appl. Math. 307 (2022) 95-101.

\bibitem{KK11}  F. Kardo\v{s}, J. Katreni\v{c}, and I. Schiermeyer, On computing the minimum $3$-path vertex cover and dissociation number of graphs, Theoret. Comput. Sci. 412 (2011) 7009-7017.

\bibitem{LX24} P.S. Li, M. Xu, The maximum number of maximum generalized $4$-independent sets in trees, J. Graph Theory 107 (2024) 359-380.

\bibitem{LS23} S.C. Li, W.T. Sun,  On the maximal number of maximum dissociation sets in forests with fixed order and dissociation number, Taiwanese J. Math. 27 (2023) 647-683.

\bibitem{LZ25} S.C. Li, Z.H. Zhou, On spectral extrema of graphs with given order and generalized $4$-independence number, Appl. Math. Comput. 484 (2025) 129018.

\bibitem{O11} Y. Orlovich, A. Dolgui, G. Finke, V. Gordon, F. Werner, The complexity of dissociation set problems in graphs, Discrete Appl. Math. 159 (13) (2011) 1352-1366.

\bibitem{T19} D. Tsur, Parameterized algorithm for $3$-path vertex cover, Theoret. Comput. Sci. 783 (2019) 1-8.


\bibitem{TZS} J.H. Tu, Z.P. Zhang, Y.T. Shi, The maximum number of maximum dissociation sets in trees, J. Graph Theory 96 (2021) 472-489.


\bibitem{Y81} M. Yannakakis, Node-deletion problems on bipartite graphs, SIAM J. Comput. 10 (1981) 310-327.


\end{thebibliography}
\end{document}